\documentclass[11pt,a4paper]{article}
\usepackage{amsmath}
\usepackage{bbm}
\usepackage{amssymb}
\usepackage{amscd}
\usepackage{graphicx}
\usepackage{caption}
\usepackage{chngcntr}
\usepackage{textcomp}
\usepackage{float}
\usepackage{subfig}
\usepackage{booktabs}
\usepackage[colorlinks,citecolor=blue,urlcolor=blue]{hyperref}
\usepackage{color}
\usepackage{amsthm}
\usepackage{enumerate}
\usepackage[round]{natbib}

\newtheorem{theorem}{Theorem}[section]
\newtheorem{remark}{Remark}[section]
\newtheorem{lemma}{Lemma}[section]
\newtheorem{proposition}{Proposition}[section]
\newtheorem{corollary}{Corollary}[section]
\newtheorem{definition}{Definition}[section]
\newtheorem{simulation}{Simulation}[section]

\counterwithin{equation}{section}

\ifdefined\final
\usepackage[disable]{todonotes}
\else
\usepackage[textsize=tiny]{todonotes}
\fi
\begin{document}

\newcommand\bbP{\mathop{\mathbb{P}}}
\newcommand\bbE{\mathop{\mathbb{E}}}
\renewcommand\hat[1]{\widehat{#1}}
\newcommand\hatE{\hat{\bbE}}
\newcommand\tI{\text{I}}
\newcommand\tII{\text{II}}
\newcommand\tIII{\text{III}}
\title{A robust  and p-hacking-proof significance test under variance uncertainty\let\thefootnote\relax\footnotetext{Li Xifeng, Yang Shuzhen, and Yao Jianfeng contributed equally to this article, the authors are listed alphabetically. We thank to Professor Shige Peng for rounds of useful discussions.}}
\author{Xifeng Li,\thanks{School of Mathematics, Shandong University, PR China, lixfay@mail.sdu.edu.cn.} \quad Shuzhen Yang,\thanks{Shandong University-Zhong Tai Securities Institute for Financial Studies, Shandong University, PR China, yangsz@sdu.edu.cn.} \quad Jianfeng Yao\thanks{Corresponding author. School of Data Science, The Chinese University of Hong Kong, Shenzhen, jeffyao@cuhk.edu.cn.}}
\date{}
\maketitle
\begin{abstract}
P-hacking poses challenges to traditional hypothesis testing. In this paper, we propose a robust method for the one-sample significance test that can protect against p-hacking from sample manipulation. Precisely, assuming a sequential arrival of the data whose variance can be time-varying and for which only lower and upper bounds are assumed to exist with possibly unknown values,  we use the modern theory of sublinear expectation to build a testing procedure which is robust under such variance uncertainty, and can protect the significance level against potential data manipulation by an experimenter. It is shown that our new method can effectively control the type I error while preserving a satisfactory  power, yet a traditional rejection criterion performs poorly under such variance uncertainty. Our theoretical results are well confirmed by a detailed simulation study.
\end{abstract}


\section{Introduction}
P-hacking refers to the manipulation of statistical analyses to achieve a desired p-value, typically below the conventional threshold of 0.05, which is often interpreted as statistically significant.    P-hacking is problematic because it can lead to false positive results that appear significant but are actually due to chance.
The phenomenon becomes more frequent recently in psychology studies, nutrition research, pharmaceutical trials, medical research, and social science surveys.     P-hacking occurs when experimenters (consciously or unconsciously) take advantage of their flexibility in data collection, analysis, and reporting.   \cite{simmons2011false} gave several examples of p-hacking.Typical p-hacking methods include
(i) Data dredging, that is, analyzing data in multiple ways and only reporting those analyses that yield significant results;
(ii) Selective reporting: only publishing results that support a hypothesis while ignoring those that do not;
(iii) Sample manipulation: changing the sample size or excluding certain data points to achieve a significant result; and
(iv) Variable manipulation: altering how variables are defined or measured to find a significant association.
 P-hacking seriously undermines the validity of scientific research and the credibility of scientific reports. For more results on p-hacking, see \cite{head2015extent}, \cite{moss2023modelling} and references therein.

In this paper, we consider a specific p-hacking problem that can be formulated through the simple one-sample significance test of the mean of a population. The difficulty we address here is that the variance of the data is highly uncertain, for which we only know the existence of a lower bound $\underline{\sigma}^{2}$ and an upper bound $\overline{\sigma}^{2}$ whose values are possibly unknown.  When the data arrive sequentially, the experimenter appears to have the opportunity to include or exclude some data points based on historical information in order to artificially inflate the significance of the test (see details in Section \ref{sec:3}).

Because of lack of information on the variance of the data which can be in particular time-varying, it is actually   difficult to construct an accurate  one-sample test that respects a preassigned significance level.  An innovative solution from this paper is that building on   the modern theory of sublinear expectation we can  design a robust rejection region that  protects against data manipulation and  guarantee a significance level close to the nominal one in the presence of a high variability of data variances.

Both theoretical analysis and simulation studies show that our procedure ensures an actual type I error rate close to the nominal one, even in the scenario where the data had been manipulated by a biased experimenter using some data selection method. Specifically, our results are obtained by considering a worst scenario among a family of distributions that accounts for the underlying variance uncertainty, and this worst scenario consideration indeed protects us against potential data manipulation.

The rest of the paper is organized as follows. Section~\ref{sec:sublin} introduces the basic concepts of the theory of sublinear expectation relevant to this paper. In Section \ref{sec:2}, we introduce our robust testing method under variance uncertainty. In Section \ref{sec:3}, we provide a theoretical analysis of the proposed method under a scenario where the experimenter is selecting the data points according to an optimal  p-hacking strategy, and compare it with the traditional one-sample test. Section \ref{sec:4} presents simulation studies to verify our theoretical results. Finally,  Section \ref{sec:5} concludes the paper with some discussions.

\section{The theory of sublinear expectation}
\label{sec:sublin}

Consider the space $\mathcal{L}_1$ of integrable and real-valued random variables
defined on a traditional probability space $(\Omega,\mathcal{F},\bbP)$. The corresponding expectation $\bbE$ is a linear operator on  $\mathcal{L}_1$ and is hereafter referred as the {\em linear expectation}. In a sense, we are considering a linear expectation space
$(\Omega,\mathcal{L}_1,\bbE)$ where $\Omega$ is the base space, $\mathcal L_1$ the space of integrable random variables and $\bbE$ the linear expectation operator. The theory of sublinear expectation can here be viewed as a sophisticated extension of the triple
$(\Omega,\mathcal{L}_1,\bbE)$  in order to capture distribution uncertainties underlying random
variables. Precisely, the extension  uses a new triple, $(\Omega,\mathcal{H},\hatE)$, replacing $(\Omega,\mathcal{L}_1,\bbE)$, where  $\mathcal H$  is an appropriate space of random variables, and $\hatE$ is the sublinear expectation acting on $\mathcal H$. Moreover, $\hatE$ is realized through the maximization:
$$
\hatE[X]=\sup\limits_{\theta \in \Theta} {\mathbb E}_{\theta}[X],\quad \forall X \in \mathcal{H},
$$
where  $\{\bbE_{\theta}\}_{\theta\in\Theta}$ is  a family of linear expectations, or linear functionals defined on $(\Omega,\mathcal{H})$.

The theory of nonlinear expectation of which sublinear expectation is a prominent offspring, was introduced in probability theory by \cite{peng2004filtration,peng2005nonlinear} to deal with model or distribution uncertainty. In particular, the theory considers a family of probability measures $\{\bbP_{\theta}\}_{\theta\in\Theta}$, with corresponding linear expectations $\{\bbE_{\theta}\}_{\theta\in\Theta}$, rather than a single probability measure $\bbP_{\theta_0}$ for the observed data. In the last two decades, the theory and methodology of sublinear expectation have been well developed into
an established area of modern probability theory, and have found multiple applications in statistics, mathematical finance, and machine learning, see, for example, \cite{epstein2013ambiguous}, \cite{lin2016k}, \cite{peng2020hypothesis}, \cite{jin2021optimal}, \cite{peng2023improving}, \cite{ji2023imbalanced}, among others. We refer to Peng's plenary talk at ICM 2010 \citep{PengICM2010} and his monograph \citep{peng2019nonlinear} for a complete account of the theory of nonlinear/sublinear expectations. For reader's convenience, the results of the theory which are relevant to the paper are summarized  in Appendix~\ref{sec:A}.

\section{The significance test under variance uncertainty}\label{sec:2}


Consider a data sequence $X_{1}, X_{2},\ldots,X_{n},\ldots$ with a common mean value $\mu\in\mathbb{R}$, while their variance is uncertain and undetermined. As said in Introduction, the framework of sublinear expectation is adopted to model such a data sequence. Precisely, the data generation process is as follows.

\medskip
\noindent\textbf{Data generation process [*]} On a probability space $(\Omega,{\mathcal F},\bbP)$, let $\varepsilon_{1},\varepsilon_{2},\ldots$ be a sequence of i.i.d. random variables such that $\bbE(\varepsilon_{i})=0$, $\bbE(\varepsilon_{i}^{2})=1$  and $\bbE(|\varepsilon_{i}|^{3})<\infty$. Let $ (\mathcal{F}_{i})_{i=0}^{\infty}$ be the natural filtration generated
by $\{\varepsilon_{i}\}_{i=1}^{\infty}$, i.e. $\mathcal{F}_{i}=\sigma(\varepsilon_{1},\ldots,\varepsilon_{i}).$  For two
given positive constants $\underline{\sigma}\le\overline{\sigma}$, denote by  $\Sigma(\underline{\sigma},\overline{\sigma})$ the family   of all
predictable sequences with respect to $(\mathcal{F}_{i})_{i=0}^{\infty}$ that always take value in $[\underline{\sigma},\overline{\sigma}]$.   For any $\{\sigma_{i}\}_{i=1}^{\infty}\in \Sigma(\underline{\sigma},\overline{\sigma})$, define $X_{i}=\sigma_{i}\varepsilon_{i}+\mu$ and $\bar{X}_{n}=(X_{1}+\ldots,X_{n})/n$. \quad$\Box$
\medskip

Suppose that the observed data sequence $X_{1},\ldots,X_{n}$ results from the above data generation process, and we are interested in the following significance test about the common mean $\mu$:
\begin{equation}\label{eq:1}
\text{(I)} \quad H_{0}:\mu\leq\mu_{0}  \quad \quad \text{versus} \quad \quad H_{1}:\mu>\mu_{0},
\end{equation}
where $\mu_{0}$ is a preassigned   constant.

\begin{remark}
Two closely related one-sample tests, numbered as {\em (II)} and {\em (III)}, are relegated to Appendix \ref{sec:B} for the clarity of the presentation.
\end{remark}

If $X_{1},\ldots,X_{n}$ are independent under the traditional probability theory $(\Omega,{\mathcal F},\bbP)$, that is, the \emph{linear expectation theory}, despite the fact that they are heteroscedastic with however uniformly bounded variances, we have, according to the classical central limit theorem and for   sufficiently large $n$:
\begin{equation}
    \label{eq:student}
\frac{\sqrt{n} (\bar{X}_{n}-\mu)}{S_{n}}~~\dot{\sim}~~N(0,1),\quad \text{where} \quad S_{n}^{2}=\frac{1}{n-1} \sum_{i=1}^{n}\left(X_{i}-\bar{X}_{n}\right)^{2}.
\end{equation}
Thus following this traditional method,   the rejection region is:
\begin{equation}\label{2}
W_{\tI}=\{\sqrt{n} (\bar{X}_{n}-\mu_{0})>S_{n}\Phi^{-1}(1-\alpha)\},
\end{equation}
where $\Phi(\cdot)$ is the cumulative distribution function of the standard normal distribution.

However, \cite{peng2020hypothesis} pointed out that by violating the independence assumption and carefully manipulating the sequence $\{\sigma_{i}\}_{i=1}^{n}$, the experimenter is able to increase the type I error rate of the procedure. In other words, the traditional rejection region $W_{\tI}$  cannot control the type I error rate when the experimenter intentionally manipulates the data. Under the sublinear expectation theory, we assume that the data has distribution uncertainty and by considering a worst scenario, we can protect the test procedure against such data cheating by the experimenter.

\subsection{The significance test under variance uncertainty with   {\em known} variance bounds   $\{\underline{\sigma}^{2},\overline{\sigma}^{2}\}$}

We first establish a central limit theorem under the sublinear expectation, which lays the foundation for our robust test.
\begin{theorem}\label{thm:1}
Let $\{X_{i}\}_{i=1}^{n}$ be given by the  data generating process [*] above. Then for any Lipschitz function $\varphi$,
\begin{equation}\label{eq:3}
\lim _{n \rightarrow \infty} \sup _{\left\{\sigma_{i}\right\} \in \Sigma(\underline{\sigma}, \overline{\sigma})} \bbE [\varphi(\sqrt{n}(\bar{X}_{n}-\mu))]=u(1,0 ; \varphi),
\end{equation}
where $\{u(t,x; \varphi):(t,x)\in [0,\infty)\times\mathbb{R}\}$ is the unique viscosity solution to the Cauchy problem,
\begin{equation}\label{eq:4}
u_t=\frac{1}{2}\left(\bar{\sigma}^2\left(u_{x x}\right)^{+}-\underline{\sigma}^2\left(u_{x x}\right)^{-}\right), \quad u(0, x)=\varphi(x).
\end{equation}
In the above expression, $u_t=\partial u/\partial t$, $u_{xx}=\partial^{2} u/\partial x^{2}$, and $a^+$ and $a^-$   denote the positive and negative parts of $a$, respectively.
\end{theorem}

For a proof of the theorem, see \cite{peng2008new}, \cite{rokhlin2015central}, or \citet[Theorem 4.1]{fang2019limit}. The theorem remains valid for any Borel-measurable indicator function $\varphi$, see \cite{peng2019nonlinear} and \cite{peng2023improving}.

Using the language of sublinear expectation, Theorem \ref{thm:1} indicates that $\sqrt{n}(\bar{X}_{n}-\mu)$ asymptotically obeys the G-normal distribution $\mathcal{N}(0,[\underline{\sigma}^{2},\overline{\sigma}^{2}])$, which reduces to the normal distribution $\mathcal{N}(0,\sigma^{2})$ when $\underline{\sigma}=\overline{\sigma}=\sigma$, that is without variance uncertainty.

\medskip
Consider the significance test \eqref{eq:1} with a known variance interval $[\underline{\sigma}^{2},\overline{\sigma}^{2}]$.  Under the null and when  $\mu=\mu_{0}$, according to Theorem \ref{thm:1}, we can calculate the asymptotic maximum false rejection probability function under the variance uncertainty as follows:
\begin{equation}\label{eq:5}
\begin{split}
p_{1}(c ; \underline{\sigma}, \overline{\sigma}):&=\lim _{n \rightarrow \infty} \sup _{\left\{\sigma_{i}\right\} \in \Sigma(\underline{\sigma}, \overline{\sigma})} \bbE \left[\mathbbm{1}\left(\sqrt{n} (\bar{X}_{n}-\mu_{0})>c\right)\right] \\
 &=\frac{2}{\overline{\sigma}+\underline{\sigma}} \int_{c}^{\infty}\{\phi(z / \overline{\sigma}) \mathbbm{1}(z \geq 0)+\phi(z / \underline{\sigma}) \mathbbm{1}(z<0)\} d z,
\end{split}
\end{equation}
where $\phi(\cdot)$ is the density function of the standard normal distribution.

It is clear that for $0<\underline{\sigma}\leq\overline{\sigma}<\infty$, $p_{1}(c ; \underline{\sigma}, \overline{\sigma})$ is a strictly decreasing function of $c$ with
$$
\lim _{c \rightarrow -\infty}p_{1}(c ; \underline{\sigma}, \overline{\sigma})=1 \qquad\text{and}\qquad \lim _{c \rightarrow \infty}p_{1}(c ; \underline{\sigma}, \overline{\sigma})=0.
$$
Therefore, $c\mapsto p_{1}(c ; \underline{\sigma}, \overline{\sigma})$ is a proper tail probability function.

We now use this asymptotic maximum false rejection probability function to define the critical value of our robust test under variance uncertainty. For a given significance level $\alpha\ (0<\alpha<0.5)$, solving the equality
$$  p_{1}(c ; \underline{\sigma}, \overline{\sigma})=\alpha
~~\Longleftrightarrow~~
\frac{2}{\overline{\sigma}+\underline{\sigma}} \int_{c}^{\infty}\phi(z / \overline{\sigma}) d z=\alpha,
$$
we find the solution
\begin{equation}
c_1=\overline{\sigma}\Phi^{-1}\left(1-\frac{\alpha(\overline{\sigma}+\underline{\sigma})}
{2\overline{\sigma}}\right).
\label{eq:c1}
\end{equation}
The rejection region of our test is finally defined as:
\begin{equation}\label{eq:6}
W_{\tI}^{G}=\{\sqrt{n} (\bar{X}_{n}-\mu_{0})>c_{1}
\}.
\end{equation}

\medskip
A few key discussions are necessary.

\begin{enumerate}
\item In case of no variance uncertainty, that is when $\underline{\sigma}=\overline{\sigma}=\sigma$, we have
$c_1= \sigma\Phi^{-1}(1-\alpha)$, and  the robust rejection region~\eqref{eq:6} becomes
\begin{equation}\label{eq:classical}
W_{\tI}^{G}=\{\sqrt{n} (\bar{X}_{n}-\mu_{0})>\sigma\Phi^{-1}(1-\alpha)
\}.
\end{equation}
We recover the classical significance test procedure with a common and known variance $\sigma^2$.

\item Under the null hypothesis $H_{0}$  with $\mu<\mu_{0}$,  we have for  large enough $n$,
\begin{equation}\nonumber
\begin{split}
&\sup _{\left\{\sigma_{i}\right\} \in \Sigma(\underline{\sigma}, \overline{\sigma})} \bbE \left[\mathbbm{1}\left(\sqrt{n} (\bar{X}_{n}-\mu_{0})>c_{1}\right)\right] \\
= &\sup _{\left\{\sigma_{i}\right\} \in \Sigma(\underline{\sigma}, \overline{\sigma})}
\bbE \left[\mathbbm{1}\left(\sqrt{n} (\bar{X}_{n}-\mu)>c_{1}+\sqrt{n}(\mu_{0}-\mu)\right)\right] \\
 \approx&\frac{2}{\overline{\sigma}+\underline{\sigma}} \int_{c_{1}+\sqrt{n}(\mu_{0}-\mu)}^{\infty}\{\phi(z / \overline{\sigma}) \mathbbm{1}(z \geq 0)+\phi(z / \underline{\sigma}) \mathbbm{1}(z<0)\} d z <\alpha.
\end{split}
\end{equation}
Moreover, by taking a sequence $\mu=\mu_{n}<\mu_{0} \text{ such that }\sqrt{n} (\mu_{0}-\mu)\rightarrow0,$ we see that
$$
\lim _{n \rightarrow \infty}\sup _{\mu<\mu_{0}}\sup _{\left\{\sigma_{i}\right\} \in \Sigma(\underline{\sigma}, \overline{\sigma})} \bbE \left[\mathbbm{1}\left(\sqrt{n} (\bar{X}_{n}-\mu_{0})>c_{1}\right)\right]=\alpha.
$$
Therefore, the rejection region $W_{\tI}^{G}$ in \eqref{eq:6} is also a rejection region for testing the null hypothesis $H_{0}:\mu\leq\mu_{0}$ under the considered variance uncertainty.

\item {\bf Power function}: under $H_1$ with  $\mu>\mu_{0}$, the power function of the robust test is
\begin{equation}\nonumber
\begin{split}
g(n,\mu)&=\inf _{\left\{\sigma_{i}\right\} \in \Sigma(\underline{\sigma}, \overline{\sigma})} \bbE \left[\mathbbm{1}\left(\sqrt{n} (\bar{X}_{n}-\mu_{0})>c_{1}\right)\right]\\
&=\inf _{\left\{\sigma_{i}\right\} \in \Sigma(\underline{\sigma}, \overline{\sigma})} \bbP \left(\sqrt{n} (\bar{X}_{n}-\mu_{0})>c_{1}\right)\\
&=1-\sup _{\left\{\sigma_{i}\right\} \in \Sigma(\underline{\sigma}, \overline{\sigma})} \bbP \left(\sqrt{n} (\bar{X}_{n}-\mu_{0})\leq c_{1}\right)\\
&=1-\sup _{\left\{\sigma_{i}\right\} \in \Sigma(\underline{\sigma}, \overline{\sigma})} \bbP \left(\sqrt{n} (\bar{X}_{n}-\mu)\leq c_{1}-\sqrt{n} (\mu-\mu_{0})\right),\\
& \simeq 1 - \frac{2}{\overline{\sigma}+\underline{\sigma}} \int_{-\infty}^{c_{1}-\sqrt{n} (\mu-\mu_{0})}\{\phi(z / \overline{\sigma}) \mathbbm{1}(z \leq 0)+\phi(z / \underline{\sigma}) \mathbbm{1}(z>0)\} d z, \\
& \longrightarrow 1 - 0 = 1, \qquad   \text{as $n\to\infty$.}
\end{split}
\end{equation}
Therefore, as the sample size grows, the power function tends to 1.
\end{enumerate}

\begin{remark} The above power calculation also shows that under a sequence of   local alternatives of the form:
\[ \mu=\mu_n >\mu_0, \quad \mu_n-\mu_0 \to 0,
\quad  \sqrt{n} (\mu_n-\mu_{0}) \to \infty,
\]
the power function $g(n,\mu_n)$ also tends to 1 as $n\to\infty$.
\end{remark}

\subsection{The significance test under variance uncertainty with   {\em unknown} variance bounds   $\{\underline{\sigma}^{2},\overline{\sigma}^{2}\}$}

When the two edges  $\{\underline{\sigma}^{2},\overline{\sigma}^{2}\}$ of the variance interval $[\underline{\sigma}^{2},\overline{\sigma}^{2}]$ are  unknown, we need  to estimate them from the data. Clearly, further information on the data generating process is needed for otherwise, it is hard to find good estimates for the two   variance bounds governing the variability of the data sequence. In this paper, we assume that the data sequence is composed with $k$ samples, a framework often assumed in the literature of distribution uncertainty as in \cite{peng2023improving}.

The following lemma is elementary and evaluates the mean and variance of a mixed group from two samples.
\begin{lemma}\label{lem:1}
If two samples have the same mean $\mu$ and different variances $(\sigma_{1}^{2},\sigma_{2}^{2})$, then the
mean of the pooled sample has average $\mu$, and its sample variance has an average in the interval $[\sigma_{1}^{2}\wedge\sigma_{2}^{2},\sigma_{1}^{2}\vee\sigma_{2}^{2}]$.
\end{lemma}

\begin{proof}
Let $X_{1},\ldots,X_{n_{1}}$ be the first sample with $\bbE(X_{i})=\mu$ and $Var(X_{i})=\sigma_{1}^{2}\ (i=1,\ldots,n_{1})$, and   $Y_{1},\ldots,Y_{n_{2}}$ be the second sample with $\bbE(Y_{j})=\mu$ and $Var(Y_{j})=\sigma_{2}^{2}\ (j=1,\ldots,n_{2})$, the two samples being independent.

Let $\bar X = n_1^{-1}\sum_iX_i$ and $\bar Y = n_2^{-1}\sum_jY_j$ be the two sample means.  The sample mean of the pooled sample  is
$$
\bar{Z}=\frac{1}{n_{1}+n_{2}} \left(\sum_{i=1}^{n_{1}}X_{i}+\sum_{j=1}^{n_{2}}Y_{j}\right)
=\frac{n_1}{n_1+n_2} \bar X+\frac{n_2}{n_1+n_2} \bar Y.
$$
Clearly $\bbE(\bar{Z})=\mu$.
The pooled sample variance is
$$
S^{2}=\frac{1}{n_{1}+n_{2}-1}\left[\sum_{i=1}^{n_{1}}(X_{i}-\bar{Z})^{2}+\sum_{j=1}^{n_{2}}(Y_{j}-\bar{Z})^{2}\right].
$$
By elementary calculation, one finds that
\begin{equation*}
\bbE(S^{2}) =\frac{n_{1}\sigma_{1}^{2}+n_{2}\sigma_{2}^{2}}{n_{1}+n_{2}}.
\end{equation*}
Evidently, this average falls in the interval $[\sigma_{1}^{2}\wedge\sigma_{2}^{2},\sigma_{1}^{2}\vee\sigma_{2}^{2}]$, and the proof is complete.
\end{proof}

\begin{remark}
Lemma \ref{lem:1} can be extended to the $k$-sample case: the mean of the pooled sample still has average $\mu$ and the pooled sample variance has an average between   $\min\{\sigma_{1}^{2},\ldots,\sigma_{k}^{2}\}$ and $\max\{\sigma_{1}^{2},\ldots,\sigma_{k}^{2}\}$.
\end{remark}

In the next theorem, we provide optimal unbiased estimators for the bounds $\overline{\sigma}^{2}$ and $\underline{\sigma}^{2}$ when the data is composed of $k$ independent samples.

\begin{theorem}\label{thm:2}
Consider a pooled sample $X_{1},\ldots,X_{n}$ made from $k$ samples of equal length $m$. For $1\leq i\leq k$,  the $i$th sample is denoted as: $X_{(i-1)*m+1},\ldots,X_{i*m}$ which are i.i.d. with mean  $\mu$  and  variance $\sigma_{i}^{2}\in[\underline{\sigma}^{2},\overline{\sigma}^{2}]$. Suppose $\underline{\sigma}^{2}=\min\limits_{1\leq i\leq k} \sigma_{i}^{2}$ and $\overline{\sigma}^{2}=\max\limits_{1\leq i\leq k}\sigma_{i}^{2}$. Consider the $k$ sub-sample means and variances: for $1\leq i \leq k$,
\begin{align*}
\bar{X}_{i}&=\frac{1}{m} \sum_{j=1}^{m}X_{(i-1)*m+j},\\
S_{ni}^{2}&=\frac{1}{m-1} \sum_{j=1}^{m}\left(X_{(i-1)*m+j}-\bar{X}_{i}\right)^{2}.
\end{align*}
 Then when $m\to\infty$,
$$
\hat{\overline{\sigma}}^{2}=\max_{1\leq i\leq k} S_{ni}^{2}
$$
is asymptotically the largest unbiased estimator for the upper variance $\overline{\sigma}^{2}$, and
$$
\hat{\underline{\sigma}}^{2}=\min_{1\leq i\leq k} S_{ni}^{2}
$$
is asymptotically the smallest unbiased estimator for the lower variance $\underline{\sigma}^{2}$.
\end{theorem}

\begin{proof}
By the law of large numbers under sublinear expectation, $S_{ni}^{2}$ obeys the maximum distribution $M_{[\underline{\sigma}^{2},\overline{\sigma}^{2}]}$ as $m\rightarrow\infty$. According to Theorem 24 in \cite{jin2021optimal}, if $Z_{1},\ldots,Z_{k}$ is an  i.i.d. sample of size $k$ from a population of maximal distribution $M_{[\underline{\mu},\overline{\mu}]}$ with unknown parameters $\underline{\mu}\leq\overline{\mu}$, where the i.i.d. condition is under the sublinear expectation, then
$$
\hat{\overline{\mu}}=\max\{Z_{1},\ldots,Z_{k}\}
$$
is the largest unbiased estimator for the upper mean $\overline{\mu}$, and
$$
\hat{\underline{\mu}}=\min\{Z_{1},\ldots,Z_{k}\}
$$
is the smallest unbiased estimator for the lower mean $\underline{\mu}$.

Applying this result to the asymptotic  maximum distribution $M_{[\underline{\sigma}^{2},\overline{\sigma}^{2}]}$ establishes the conclusions.
\end{proof}

In practice, the $k$ sub-sample sizes may not be equal, and we do not know the exact change points of the sub-samples. We can still construct optimal estimators for $\overline{\sigma}^{2}$ and $\underline{\sigma}^{2}$ using the {\em tool of  moving blocks}. The data $X_{1},\ldots,X_{n}$
are scanned sequentially into $L=n-m+1$ overlapping moving blocks of a given block length $m$:
$$
\{X_{1},\ldots X_{m}\},\{X_{2},\ldots,X_{m+1}\},\ldots,\{X_{n-m+1},\ldots,X_{n}\}.
$$
For $1\leq l \leq L$, denote the data in the $l$th block by $B_{l}=\{X_{j}\}_{l\leq j\leq l+m-1}$ with sample mean and sample variance:
$$
\bar{X}_{nl}=\frac{1}{m} \sum_{j=l}^{l+m-1}X_{j},
\qquad
\tilde S_{nl}^{2}=\frac{1}{m-1} \sum_{j=l}^{l+m-1}\left(X_{j}-\bar{X}_{nl}\right)^{2}.
$$
Define the {\em moving-block estimators}:
\begin{equation}
    \tilde{\overline{\sigma}}^{2}=\max_{1\leq l\leq L} \tilde S_{nl}^{2},
    \qquad
    \tilde{\underline{\sigma}}^{2}=\min_{1\leq l\leq L} \tilde S_{nl}^{2}.
    \label{eq:mb-estimators}
\end{equation}
\begin{corollary}\label{cor:mb}
Under the conditions of Theorem \ref{thm:2}, we further assume: (i) the sizes of the samples are unequal but greater than the block length $m$; (ii) we do not know the number of the samples, (iii)  we do not know the exact locations of the change points of the sub-samples. Then the moving-block estimators
$\tilde{\overline{\sigma}}^{2}$ and $\tilde{\underline{\sigma}}^{2}$
in ~\eqref{eq:mb-estimators}
are, asymptotically, the largest unbiased estimator for the upper variance $\overline{\sigma}^{2}$, and
the smallest unbiased estimator for the lower variance $\underline{\sigma}^{2}$, respectively.
\end{corollary}

\begin{proof}
If one block $B_{l}$ overlaps with two groups of samples, by Lemma \ref{lem:1} the variance of the block is between the variances of these two groups of samples. In addition, for each sample there is a block $B_{l}$ that is completely contained in this sample. Combining with Theorem \ref{thm:2} establishes the result.
\end{proof}
\begin{remark}
The use of moving blocks  to estimate the upper and lower variances $\overline{\sigma}^{2}$ and $\underline{\sigma}^{2}$ was first proposed by \cite{peng2023improving}. However, there was no a rigorous proof for the properties of the estimators as established in  Corollary~\ref{cor:mb}.
\end{remark}
\section{Performance against p-hacking}\label{sec:3}

We now  analyze the robustness of the proposed test in a p-hacking scenario where an adversarial experimenter would manipulate the data to gain certain benefits. Specifically, consider the one-sample significance test about a population mean   $\mu$ which, for example, can be regarded as the efficacy of a new drug that can only be put into the market when $\mu>\mu_{0}$ for some reference value $\mu_0$. If the experimenter can reject the null hypothesis $H_{0}:\mu\leq\mu_{0}$ based on the experimental data, he will gain expected benefits, otherwise he may suffer losses due to the initial investment costs. Note that there is no need for the experimenter to spend effort to manipulate the data when $\mu>\mu_{0}$ while when the data is close to the null hypothesis $\mu\le\mu_{0}$,    a dishonest experimenter could  manipulate the data, through data selection for example,  in order to artificially increase the significance of the drug.  Therefore for the sake of public interest, regulatory authorities need to take necessary measures to counter such data manipulation by dishonest experimenters and ensure that the actual type I error rate does not artificially exceed the nominal significance level $\alpha$.

In Section \ref{sec:3.1}, we present an optimal manipulation strategy that the experimenter will adopt to maximize the probability of $\{\sqrt{n} (\bar{Z}_{n}-\mu_{0})>c\}$ when $\mu=\mu_{0}$, where $c$ is a given constant. Then in Section \ref{sec:3.2}, we compare the defense capability of the traditional test procedure with our robust procedure. Although the nominal significance level is $\alpha$, the experimenter can use an optimal manipulation strategy to select data and increases the actual type I error rate of the traditional rejection region $W_{\tI}$ to a value higher than  $2\alpha\overline{\sigma}/(\underline{\sigma}+\overline{\sigma})$, which is always larger than $\alpha$.  In contrast, under the optimal data selection strategy for p-hacking,  the actual type I error rate of our robust test with rejection region $W_{\tI}^{G}$ will stay close to the nominal significance level $\alpha$.
\subsection{An optimal data manipulation strategy for p-hacking} \label{sec:3.1}
For simplicity, here we only discuss the   case with two sub-samples. Assuming that there are two sub-samples that have different variances but the same mean: the first sample $\{W_{1i}\}_{i=1}^{n_{1}}$ obeys $\mathcal{N}(\mu,\overline{\sigma}^{2})$ and the second sample $\{W_{2i}\}_{i=1}^{n_{2}}$ obeys $\mathcal{N}(\mu,\underline{\sigma}^{2})$. Let $n=n_{1}\wedge n_{2}$.  For a given constant $c$ the  experimenter strategically chooses $Z_i=W_{1i}$, or $Z_i=W_{2i}$ at the $i$th stage to maximize the probability of $\{\sqrt{n} (\bar{Z}_{n}-\mu_{0})>c\}$ when $\mu=\mu_{0}$ (i.e.,  maximize the actual type I error rate). Until the final stage $n$, the experimenter follows a sequential strategy $\theta=(\sigma_{1},\ldots,\sigma_{n})\in\{\underline{\sigma},\overline{\sigma}\}^n$, where $\sigma_{i}=\overline{\sigma}$ indicates $Z_i=W_{1i}$,  and $\sigma_{i}=\underline{\sigma}$ indicates $Z_i=W_{2i}$. In other words,
$$
Z_i=\left\{\begin{array}{l}
W_{1i}, \quad \text { if } \sigma_{i}=\overline{\sigma}, \\[1mm]
W_{2i}, \quad \text { if } \sigma_{i}=\underline{\sigma},
\end{array}\right.
$$
where the $\{\sigma_{i}\}_{i=1}^{n}$ are practically identified using historical information.

In the next theorem, we present an optimal manipulation strategy $\theta^*$ the experimenter can  take to maximize the probability of rejecting the null hypothesis when the data is actually generated under the null hypothesis.

\begin{theorem}\label{thm:3}
(Optimal manipulation strategy $\theta^*$) For a given constant $c$, an experimenter can construct the asymptotically optimal strategy $\theta^*$ to maximize the probability of $\{\sqrt{n} (\bar{Z}_{n}-\mu_{0})>c\}$ when $\mu=\mu_{0}$ (i.e., achieving the supremum in \eqref{eq:5}) as follows:  choose an arbitrary $\sigma_{1}$ and for $i\geq2$, choose
\begin{equation}\label{7}
\begin{array}{rlrl}
\sigma_{i}=\overline{\sigma}, & & \text { if }\
   \frac1{\sqrt{n}}\xi_{i-1}  \leq c, \\
\sigma_{i}=\underline{\sigma}, & & \text { if }\   \frac1{\sqrt{n}}\xi_{i-1} >c, \\
\end{array}
\end{equation}
where $\xi_i =\sum_{\ell=1}^i(Z_{\ell}-\mu_{0})$.
\end{theorem}

\begin{proof}
By the Appendix of \cite{peng2020hypothesis}, the asymptotically optimal strategy that attains the supremum in \eqref{eq:3} is, for $i=2,\ldots,n$,
\begin{equation}\label{8}
\begin{array}{rlrl}
\sigma_{i}=\overline{\sigma}, & & \text { if }\ u_{xx}\left(1-(i-1)/n, \frac1{\sqrt{n}}\xi_{i-1}\right)\geq0, \\
\sigma_{i}=\underline{\sigma}, & & \text { if }\  u_{xx}\left(1-(i-1)/n, \frac1{\sqrt{n}}\xi_{i-1}\right)<0, \\
\end{array}
\end{equation}
 where $u$ is the solution to the G-heat equation \eqref{eq:4}.

For $\varphi(x)=\mathbbm{1}(x>c)$, $u$ is given by $u(t, x) = f\{(x - c)/\sqrt{t}\}$ where
$$
f(y)=f(y ; \underline{\sigma}, \overline{\sigma})=\frac{2}{\overline{\sigma}+\underline{\sigma}} \int_{-y}^{\infty}\{\phi(z / \overline{\sigma}) \mathbbm{1}(z \geq 0)+\phi(z / \underline{\sigma}) \mathbbm{1}(z<0)\} d z
$$
Direct calculations give that
$$
\begin{aligned}
&u_{x x}(t, x) =\frac{1}{t} f_{y y}\left(\frac{x-c}{\sqrt{t}}\right),  \\
&f_{y y}(y) =\frac{-2 y}{\bar{\sigma}+\underline{\sigma}}\left\{\frac{1}{\bar{\sigma}^2} \phi(y / \bar{\sigma}) \mathbbm{1}(y \leq 0)+\frac{1}{\underline{\sigma}^2} \phi(y / \underline{\sigma}) \mathbbm{1}(y>0)\right\} .
\end{aligned}
$$
It is obvious that $u_{xx}(t,x) \geq 0$ is equivalent to $x \leq c$ and $u_{xx}(t,x) < 0$ is equivalent to $x > c$. The proof is then complete.
\end{proof}

\begin{remark}
It can be seen from the proof of Theorem \ref{thm:3} that the theorem still holds in the $k$-sample case.
\end{remark}

\subsection{Comparison with the traditional test procedure}\label{sec:3.2}
Here we compare the theoretical performance of our test with the traditional procedure under the scenario that the experimenter manipulates the data using the optimal strategy found in Theorem~\ref{thm:3}. The following corollary reveals the failure of the traditional test in~\eqref{2} in  controlling the actual type I error rate.

\begin{proposition}[\cite{peng2020hypothesis}]\label{prop:1}
For a given nominal significance level $\alpha\ (0<\alpha<0.5)$, if the experimenter chooses data $Z_{1},\ldots,Z_{n}$ as in Theorem \ref{thm:3} with $c=\bar{\sigma} \Phi^{-1}\left(1-\alpha\right)$, then the actual type I error rate of the rejection region given by the traditional test procedure is at least $2\alpha\overline{\sigma}/(\underline{\sigma}+\overline{\sigma})$ (which is larger than $\alpha$ unless $\overline{\sigma}=\underline{\sigma}$).
\end{proposition}

\begin{proof}
For a given nominal significance level $\alpha$ and a group of data $Z_{1},\ldots,Z_{n}$, the rejection region given by the traditional linear expectation is, see Equation \eqref{2},
$$
W_{\tI}=\{\sqrt{n} (\bar{Z}_{n}-\mu_{0})>S_{n}\Phi^{-1}(1-\alpha)
\}.
$$
If the experimenter selects the data $Z_{1},\ldots,Z_{n}$ according to  the
optimal strategy in Theorem \ref{thm:3} with $c=\bar{\sigma} \Phi^{-1}\left(1-\alpha\right)$, observing that $S_{n}\leq\overline{\sigma}$ a.e. for large enough $n$, then when $\mu=\mu_{0}$ we can obtain
\begin{equation}
\begin{split}
&\lim_{n\to \infty} \bbE \left[\mathbbm{1}\left(\sqrt{n} (\bar{Z}_{n}-\mu_{0}) > S_{n} \Phi^{-1}\left(1-\alpha\right)\right)\right]\\
\geq&\lim_{n\to \infty} \bbE \left[\mathbbm{1}\left(\sqrt{n} (\bar{Z}_{n}-\mu_{0}) >\bar{\sigma} \Phi^{-1}\left(1-\alpha\right)\right)\right]\\
 =&p_{1}(\bar{\sigma} \Phi^{-1}\left(1-\alpha\right) ; \underline{\sigma}, \overline{\sigma})=\frac{2\alpha\overline{\sigma}}{\overline{\sigma}+\underline{\sigma}}.
\end{split}
\end{equation}
The proof is complete.
\end{proof}

\begin{remark}
Proposition~\ref{prop:1} shows that the p-hacking of the experimenter is successful in the sense that the test statistic has been artificially inflated by data selection which leads to an actual type I error larger than the expected nominal significance level $\alpha$. Such data cheating has no major impact if the data has intrinsic significance, e.g., an actually efficient new drug. In the opposite situation where no such intrinsic significance exists, the data cheating can report a fake significance with an inflated test statistic.
\end{remark}

\begin{remark}
The result in Proposition~\ref{prop:1} was established  in \cite{peng2020hypothesis}, and we restate it here for the convenience of readers. It is worth noting that when $\overline{\sigma}$ is very large relative to $\underline{\sigma}$, the actual type I error rate of the traditional rejection region $W_{\tI}$ is approximately equal to $2\alpha$ by selecting  the data sequence as in Theorem~\ref{thm:3}.
\end{remark}

  We now theoretically prove the success of our robust test in controlling the actual type I error rate with the p-hacked data.

\begin{proposition}\label{prop:2}
For a given nominal significance level $\alpha\ (0<\alpha<0.5)$ and with the p-hacked data  $Z_{1},\ldots,Z_{n}$ selected as in Theorem \ref{thm:3} with $c=c_1=\overline{\sigma}\Phi^{-1}\left(1-\frac{\alpha(\overline{\sigma}+\underline{\sigma})}{2\overline{\sigma}}\right)$, the actual asymptotic type I error rate of the robust test given by the sublinear expectation in \eqref{eq:6} equals $\alpha$.
\end{proposition}

\begin{proof}
For a given nominal significance level $\alpha$ and a group of data $Z_{1},\ldots,Z_{n}$, the rejection region of our robust test  given by the sublinear expectation is, see equation \eqref{eq:6},
$$
W_{\tI}^{G}=\{\sqrt{n} (\bar{Z}_{n}-\mu_{0})>c_{1}
\},\ \text{with} \ c_{1}=\overline{\sigma}\Phi^{-1}\left(1-\frac{\alpha(\overline{\sigma}+\underline{\sigma})}{2\overline{\sigma}}\right).
$$
Assume that the experimenter chooses data $Z_{1},\ldots,Z_{n}$ according to the optimal strategy in Theorem \ref{thm:3} with $c=c_{1}$, when $\mu=\mu_{0}$ we have

\[\lim_{n\to \infty} \bbE \left[\mathbbm{1}\left(\sqrt{n} (\bar{Z}_{n}-\mu_{0}) > c_1\right)\right]=p_{1}( c_1 ; \underline{\sigma}, \overline{\sigma})=\alpha.\]
The type I error rate of the robust test is controlled.
\end{proof}

\begin{remark}
For our robust rejection region $W_{\tI}^{G}$, that is $\{\sqrt{n} (\bar{Z}_{n}-\mu_{0})>c_1\}$, if the experimenter selects data $Z_{1},\ldots,Z_{n}$ according to equation \eqref{7} with $c=c_{*}>0$, $c_{*}\neq c_1$, when $\mu=\mu_{0}$ we have
\begin{equation}\nonumber
\begin{split}
&\lim_{n\to \infty} \bbE \left[\mathbbm{1}\left(\sqrt{n} (\bar{Z}_{n}-\mu_{0}) > c_1\right)\right]\\
 \leq&\lim _{n \rightarrow \infty} \sup _{\left\{\sigma_{i}\right\} \in \Sigma(\underline{\sigma}, \overline{\sigma})} \bbE  \left[\mathbbm{1}\left(\sqrt{n} (\bar{Z}_{n}-\mu_{0}) > c_1\right)\right]\\
=& p_{1}( c_1 ; \underline{\sigma}, \overline{\sigma})=\alpha.
\end{split}
\end{equation}
Therefore, for any selection strategy $\{\sigma_{i}\}_{i=1}^{n}\in \Sigma(\underline{\sigma},\overline{\sigma})$, the robust test given by the sublinear expectation in \eqref{eq:6}  can ensure that the actual asymptotic type I error rate does not exceed $\alpha$.
\end{remark}

Obviously, for sufficiently large $n$
$$
(c_1,\infty)\subset (S_{n} \Phi^{-1}(1-\alpha),\infty), a.e..
$$
Thus, if the experimenter chooses data $Z_{1},\ldots,Z_{n}$ according to the optimal strategy and constructs the rejection region $(S_{n} \Phi^{-1}(1-\alpha),\infty)$ based on classical central limit theorem, the experimenter will more easily   reject the null hypothesis. This  does not happen when the rejection region is $(c_1,\infty)$, which avoids being cheated by the dishonest experimenter.

\section{Monte-Carlo study}\label{sec:4}
In this section, we present some simulation results to comprehensively compare the finite sample performance of our robust significance test  in Section \ref{sec:2} with the traditional method. Our new method presumes that the data have variance uncertainty, however, the traditional method can not recognize such variance uncertainty and blindly believes that the data are independent and identically distributed. To make clear comparisons, we consider the empirical type I error rates and the empirical powers to assess the two methods.

For each simulation, we assume that $\underline{\sigma}$ and $\overline{\sigma}$ are known and  use data $Z_{1},\ldots,Z_{n}$ to conduct the testing problem \eqref{eq:1} under classical linear expectation and sublinear expectation, respectively. The empirical type I error rates and the empirical powers are calculated from 5,000 repetitions. Recall that the rejection region under the classical linear expectation is, see Equation~\eqref{2}:
$$
W_{\tI}=\{\sqrt{n} (\bar{Z}_{n}-\mu_{0})> S_{n}\Phi^{-1}(1-\alpha)
\}.
$$
The rejection region under the sublinear expectation is, see equation \eqref{eq:6}:
$$
W_{\tI}^{G}=\{\sqrt{n} (\bar{Z}_{n}-\mu_{0})>c_{1}
\},\text{ with } \ c_{1}=\overline{\sigma}\Phi^{-1}\left(1-\frac{\alpha(\overline{\sigma}+\underline{\sigma})}{2\overline{\sigma}}\right).
$$

The two simulation experiments have almost identical settings; the only difference is that the data selection threshold $c$ in equation \eqref{7} is different. In the first experiment, $c=\bar{\sigma} \Phi^{-1}\left(1-\alpha\right)$, which corresponds to the optimal p-hacking strategy to achieve the maximum false rejection probability for the traditional test, see Proposition~\ref{prop:1}. In the second experiment, $c=c_1$, which is the optimal value that maximizes the false rejection probability of our robust test using hacked data, see Proposition~\ref{prop:2}.

\begin{simulation}\label{sim:1}
Assume that the first sample $\{W_{1i}\}_{i=1}^{n}$ obeys $\mathcal{N}(\mu,\overline{\sigma}^{2})$ and the second sample $\{W_{2i}\}_{i=1}^{n}$ obeys $\mathcal{N}(\mu,\underline{\sigma}^{2})$, where $\underline{\sigma}=0.5$ and $\overline{\sigma}=1$. Let $\mu_{0}=0$ and the significance level $\alpha=0.05$, we obtain the optimal strategy $\theta_{1}^{*}=(\sigma_{1},\ldots,\sigma_{n})$ according to Theorem \ref{thm:3} with $c=\bar{\sigma} \Phi^{-1}\left(1-\alpha\right)$ and define
$$
Z_i=\left\{\begin{array}{l}
W_{1i}, \text { if } \sigma_{i}=\overline{\sigma}, \\[1mm]
W_{2i}, \text { if } \sigma_{i}=\underline{\sigma}.
\end{array}\right.
$$
\end{simulation}

\begin{table}[hb]
\setlength{\abovecaptionskip}{0.1cm}
\setlength{\belowcaptionskip}{0.2cm}
\centering
\caption{Empirical type I error rates ($\mu=\mu_{0}$) over 5,000 repetitions for Simulation \ref{sim:1}.}
\label{table:1}
{\small
\begin{tabular}{lcc}
\toprule
&robust test&classical test\\
\midrule
n=50&0.0448&0.0730\\
n=100&0.0458&0.0744\\
n=150&0.0452&0.0770\\
n=200&0.0390&0.0670\\
n=300&0.0394&0.0682\\
n=400&0.0430&0.0746\\
n=500&0.0364&0.0698\\
n=600&0.0396&0.0746\\
n=700&0.0412&0.0778\\
n=800&0.0368&0.0678\\
n=900&0.0372&0.0688\\
n=1000&0.0444&0.0722\\
\bottomrule
\end{tabular}
}
\end{table}

\begin{figure}[ht]
\setlength{\abovecaptionskip}{0.1cm}
\setlength{\belowcaptionskip}{0.2cm}
\centering
\includegraphics[width=0.6\textwidth]{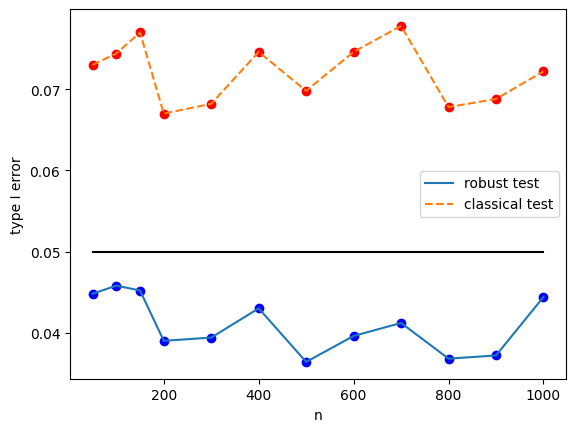}
\caption{Empirical type I error rate ($\mu=\mu_{0}$) plot over 5,000 repetitions for Simulation \ref{sim:1}}
\label{fig:1}
\end{figure}

\begin{figure}[ht]
\setlength{\abovecaptionskip}{0.1cm}
\setlength{\belowcaptionskip}{0.2cm}
\centering
\subfloat{\includegraphics[width=0.5\textwidth]{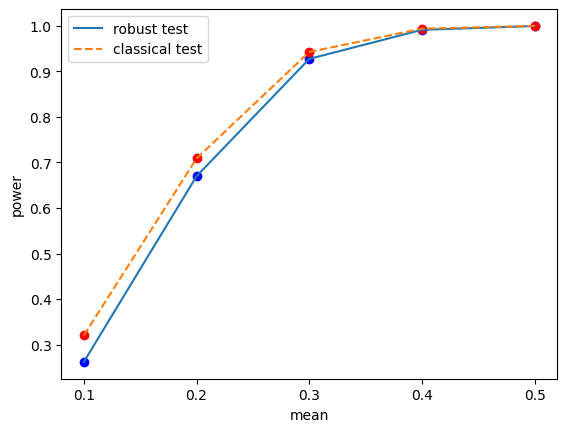}}
~~
\subfloat{\includegraphics[width=0.5\textwidth]{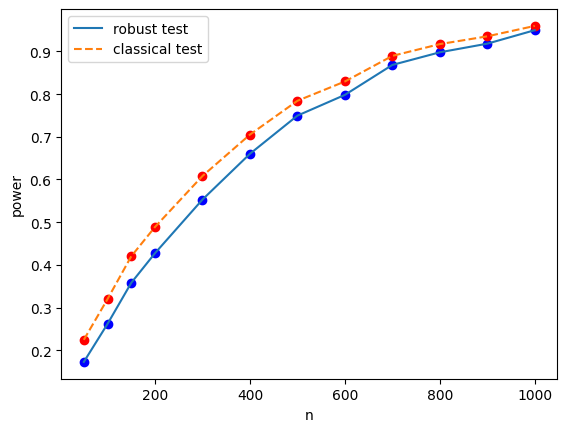}}
\caption{Empirical power plots over 5,000 repetitions for Simulation \ref{sim:1} with n=100 and varying $\mu$ (left panel), and with $\mu$=0.1 and varying $n$ (right panel).}
\label{fig:2}
\end{figure}

Simulation \ref{sim:1} corresponds to Proposition~\ref{prop:1}, where the experimenter manipulates the data to maximize the type I error rate of the traditional rejection region $W_{\tI}$. The results associated with this simulation are reported in Table \ref{table:1}, Figures \ref{fig:1} and   \ref{fig:2}. From Table \ref{table:1} and Figure \ref{fig:1} we can see that the empirical type I error rate ($\mu=\mu_{0}$) of our robust test is slightly less than 0.05, while the one of the classical test is significantly greater than 0.05 and does not approach 0.05 as $n$ increases: indeed the error rate  inflation is constantly between 35\% and 55\%! Note also that the empirical type I error rate ($\mu=\mu_{0}$) of the classical test is greater than or equal to $2\alpha\overline{\sigma}/(\underline{\sigma}+\overline{\sigma})\approx0.067$ for all $n$, as predicted by the theoretical result in Proposition~\ref{prop:1}. In Figure \ref{fig:2}, when n=100 and $\mu$ increases (left panel), the power increases with $\mu$ under both methods, and the difference between the two becomes smaller and smaller. When $\mu$=0.1 and $n$ increases (right panel), similar conclusions hold.

\begin{simulation}\label{sim:2}
Assume that the first sample $\{W_{1i}\}_{i=1}^{n}$ obeys $\mathcal{N}(\mu,\overline{\sigma}^{2})$ and the second sample $\{W_{2i}\}_{i=1}^{n}$ obeys $\mathcal{N}(\mu,\underline{\sigma}^{2})$, where $\underline{\sigma}=0.5$ and $\overline{\sigma}=1$. Let $\mu_{0}=0$ and the significance level $\alpha=0.05$, we obtain the optimal strategy $\theta_{2}^{*}=(\sigma_{1},\ldots,\sigma_{n})$ according to Theorem \ref{thm:3} with $c=c_1=\overline{\sigma}\Phi^{-1}\left(1-\frac{\alpha(\overline{\sigma}+\underline{\sigma})}{2\overline{\sigma}}\right)$ and define
$$
Z_i=\left\{\begin{array}{l}
W_{1i}, \text { if } \sigma_{i}=\overline{\sigma}, \\[1mm]
W_{2i}, \text { if } \sigma_{i}=\underline{\sigma}.
\end{array}\right.
$$
\end{simulation}

Simulation \ref{sim:2} corresponds to Proposition~\ref{prop:2}, where the experimenter manipulates the data to maximize the type I error rate of the robust rejection region $W_{\tI}^{G}$. The simulation results associated with this simulation are summarized in Table \ref{table:2}, Figures \ref{fig:3} and   \ref{fig:4}. Table \ref{table:2} and Figure \ref{fig:3} show that the empirical type I error rate ($\mu=\mu_{0}$) of the robust test is around 0.05, which validates the theoretical result in Proposition \ref{prop:2}. Again, the empirical type I error rate ($\mu=\mu_{0}$) of the classical test is much larger than 0.05 even if $n=1000$, with about 20\%-30\% inflation. For the empirical powers, Figure \ref{fig:4} indicates that the results are similar to those in Simulation \ref{sim:1}.

\begin{table}[ht]
\setlength{\abovecaptionskip}{0.1cm}
\setlength{\belowcaptionskip}{0.2cm}
\centering
\caption{Empirical type I error rates ($\mu=\mu_{0}$) over 5,000 repetitions for Simulation \ref{sim:2}}
\label{table:2}
{\small
\begin{tabular}{lcc}
\toprule
&robust test&classical test\\
\midrule
n=50&0.0492&0.0658\\
n=100&0.0500&0.0630\\
n=150&0.0494&0.0644\\
n=200&0.0506&0.0644\\
n=300&0.0502&0.0650\\
n=400&0.0508&0.0634\\
n=500&0.0506&0.0636\\
n=600&0.0492&0.0612\\
n=700&0.0514&0.0666\\
n=800&0.0512&0.0640\\
n=900&0.0468&0.0586\\
n=1000&0.0516&0.0620\\
\bottomrule
\end{tabular}
}
\end{table}

\begin{figure}[hb!]
\setlength{\abovecaptionskip}{0.1cm}
\setlength{\belowcaptionskip}{0.2cm}
\centering
\includegraphics[width=0.6\textwidth]{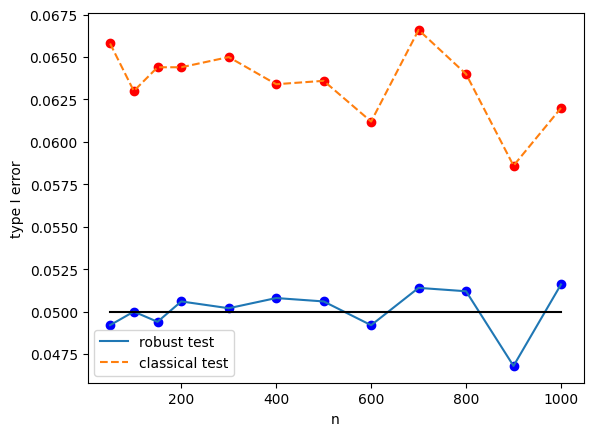}
\caption{Empirical type I error rate ($\mu=\mu_{0}$) plot over 5,000 repetitions for Simulation \ref{sim:2}}
\label{fig:3}
\end{figure}

\begin{figure}[hb!]
\setlength{\abovecaptionskip}{0.1cm}
\setlength{\belowcaptionskip}{0.2cm}
\centering
\subfloat{\includegraphics[width=0.5\textwidth]{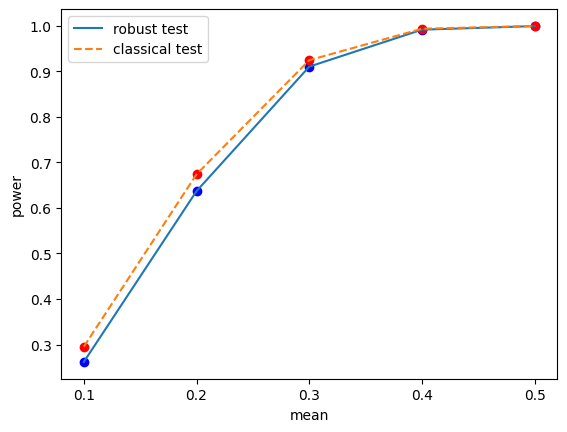}}
~~
\subfloat{\includegraphics[width=0.5\textwidth]{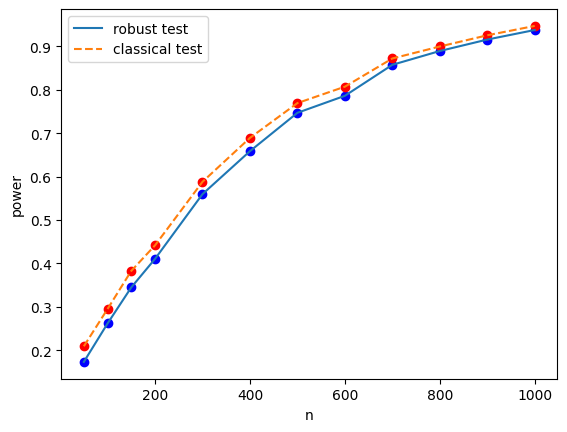}}
\caption{Empirical power plots over 5,000 repetitions for Simulation \ref{sim:2} with n=100 and varying $\mu$ (left panel), and with $\mu$=0.1 and varying $n$ (right panel).}
\label{fig:4}
\end{figure}

In both simulations, our new method effectively controls the type I error rate, whereas the traditional method fails to control it which provides an opportunity for the experimenter to cheat by manipulating the data. Although the new method suffers from low power when $n$ and $\mu$ are both very small, it achieves high power as $n$ or $\mu$ increases. These findings confirm  well our  theoretic results. In conclusion, the two simulation studies demonstrate that our proposed method outperforms the traditional method when facing p-hacked data thanks to its   robustness and  satisfactory test power.

\section{Discussions}\label{sec:5}
Using the modern theory of sublinear expectation to deal with distribution uncertainty,  we have developed a robust procedure for the traditional significance test. The procedure is superior to the traditional approach when facing p-hacked data like in the situation analyzed  by \cite{peng2020hypothesis} where data has variance uncertainty and  a dishonest experimenter can select the data sequence using an optimal strategy
to artificially inflate the rejection probability. Particularly in this situation,  the traditional procedure will report an inflated and wrong test significance while our new method will resist to such wrong inflation without loosing much test power.

Several interesting directions are worth investigation in the future. First, one may consider a robust significance test under variance uncertainty in a multivariate or even high-dimensional situation.  Another direction is to explore the significance test for linear regressions under model or noise distribution uncertainty.

\newpage
\appendix
\section{Basic concepts of sublinear expectation}\label{sec:A}

As theory of sublinear expectation is recent and technical,
this section introduces some basic concepts of the theory to help the understanding  of our study.   We refer to \cite{peng2019nonlinear} for a book-length introduction of the theory.

\begin{definition}
Let $\Omega$ be a given set and let $\mathcal{H}$ be a linear space of real valued functions defined on $\Omega$. we suppose that $\mathcal{H}$ satisfies:
\begin{enumerate}[(i)]
  \item $c \in \mathcal{H}$ for each constant $c$;
  \item $|X| \in \mathcal{H}$ if $X \in \mathcal{H}$.
\end{enumerate}
The space $\mathcal{H}$ can be considered as the \emph{space of random variables}.
\end{definition}

\begin{definition}
A \emph{sublinear expectation} $\hat{\mathbb{E}}$ is a functional $\hat{\mathbb{E}}:\mathcal{H}\rightarrow \mathbb{R}$ satisfying:
\begin{enumerate}[(i)]
  \item Monotonicity: $\hat{\mathbb{E}}[X]\geq \hat{\mathbb{E}}[Y]$ if $X\geq Y$;
  \item Constant preserving: $\hat{\mathbb{E}}[c]=c,\  \forall c \in \mathbb{R}$;
  \item Sub-additivity: $\forall X,Y \in \mathcal{H},\ \hat{\mathbb{E}}[X+Y]\leq \hat{\mathbb{E}}[X]+\hat{\mathbb{E}}[Y]$;
  \item Positive homogeneity: $\hat{\mathbb{E}}[\lambda X]=\lambda \hat{\mathbb{E}}[X],\ \forall \lambda\geq0$.
\end{enumerate}
The triple $(\Omega,\mathcal{H},\hat{\mathbb{E}})$ is called a \emph{sublinear expectation space}.
\end{definition}

\begin{definition}
A sublinear expectation $\hat{\mathbb{E}}$ defined on $(\Omega,\mathcal{H})$ is said to be regular if
$$
\hat{\mathbb{E}}[X_{n}]\rightarrow 0, \quad n\to\infty,
$$
for each sequence $\{X_{n}\}_{n=1}^{\infty}$ of random variables in $\mathcal{H}$ such that $X_{n}(\omega)\downarrow0$ for each $\omega \in \Omega$.
\end{definition}

\begin{theorem}
(Representation of a sublinear expectation) Let $\hat{\mathbb{E}}$ be a sublinear expectation defined on $\mathcal{H}$, then there exists a family of linear functionals $\{E_{\theta}:\theta\in\Theta\}$ defined on $(\Omega,\mathcal{H})$ such that
$$
\hat{\mathbb{E}}[X]=\sup \limits_{\theta \in \Theta} E_{\theta}[X],\quad \forall X \in \mathcal{H}.
$$
\end{theorem}

\begin{theorem}
(Robust Daniell-Stone Theorem) Assume that $(\Omega,\mathcal{H},\hat{\mathbb{E}})$ is a sublinear expectation space. If $\hat{\mathbb{E}}$ is regular, then there exists a class of probability measures $\{P_{\theta}\}_{\theta\in\Theta}$ on $(\Omega,\sigma(\mathcal{H}))$ such that
$$
\hat{\mathbb{E}}[X]=\sup \limits_{\theta \in \Theta} \int_{\Omega}X(\omega)dP_{\theta}, \quad \text{ for each } X \in \mathcal{H},
$$
where $\sigma(\mathcal{H})$ is the smallest $\sigma$-algebra generated by $\mathcal{H}$.
\end{theorem}

\begin{definition}
(Distribution equality) Let $X$ and $Y$ be two random variables defined on a sublinear expectation space $(\Omega,\mathcal{H},\hat{\mathbb{E}})$. They are called identically distributed, denoted by $X\overset{\text{d}}{=}Y$, if
$$
\hat{\mathbb{E}}[\varphi(X)]=\hat{\mathbb{E}}[\varphi(Y)],\quad \forall \varphi\in \emph{C}_{Lip}(\mathbb{R}).
$$
\end{definition}

\begin{remark}
The distribution of $X\in \mathcal{H}$ has the following typical parameters:
$$
\overline{\mu}:=\hat{\mathbb{E}}[X],\quad \underline{\mu}:=-\hat{\mathbb{E}}[-X].
$$
The interval $[\underline{\mu},\overline{\mu}]$ characterizes the mean-uncertainty of $X$.
\end{remark}

\begin{definition}
(Independence) In a sublinear expectation space $(\Omega,\mathcal{H},\hat{\mathbb{E}})$, a random vector $Y\in \mathcal{H}^{n}$ is said to
be independent of another random vector $X\in \mathcal{H}^{m}$ under $\hat{\mathbb{E}}$, if for each test function $\varphi\in \emph{C}_{Lip}(\mathbb{R}^{m+n})$ we have
$$
\hat{\mathbb{E}}[\varphi(X,Y)]=\hat{\mathbb{E}}[\hat{\mathbb{E}}[\varphi(x,Y)]_{x=X}].
$$
\end{definition}

\begin{remark}
Under sublinear expectation, independence is not  symmetric: "$Y$ is independent of $X$" does not imply automatically that "$X$ is
independent of $Y$". Example 1.3.15 of \cite{peng2019nonlinear} provides such an example.
\end{remark}

\begin{definition}
(i.i.d. sequence) A sequence of random variables $\{X_{i}\}_{i=1}^{\infty}$ is said to be i.i.d., if for each $i=1,2,\ldots, X_{i+1}$ is identically distributed as $X_{1}$ and independent of $(X_{1},\ldots,X_{i})$.
\end{definition}

\begin{definition}
(Maximal distribution) A random variable $\eta$ on a sublinear expectation space $(\Omega,\mathcal{H},\hat{\mathbb{E}})$ is called maximal distributed if
$$
\hat{\mathbb{E}}[\varphi(\eta)]=\sup \limits_{\underline{\mu}\leq y \leq\overline{\mu}}\varphi(y),\qquad \forall \varphi\in \emph{C}_{Lip}(\mathbb{R}),
$$
where $\overline{\mu}=\hat{\mathbb{E}}[\eta] \text{ and } \underline{\mu}=-\hat{\mathbb{E}}[-\eta].$ This distribution is denoted by $\eta\overset{\text{d}}{=}M_{[\underline{\mu},\overline{\mu}]}$.
\end{definition}

\begin{theorem}
(Law of large numbers) Let $\{X_{i}\}_{i=1}^{\infty}$ be a sequence of i.i.d. random variables on a sublinear expectation space $(\Omega,\mathcal{H},\hat{\mathbb{E}})$. We assume that $\hat{\mathbb{E}}[|X_{1}|^{2}]<\infty$, then for any Lipschitz function $\varphi$,
$$
\lim_{n\rightarrow\infty} \hat{\mathbb{E}}[\varphi(\frac{1}{n}\sum_{i=1}^{n}X_{i})]=\sup \limits_{\underline{\mu}\leq y \leq\overline{\mu}}\varphi(y),
$$
where $\overline{\mu}=\hat{\mathbb{E}}[X_{1}] \text{ and } \underline{\mu}=-\hat{\mathbb{E}}[-X_{1}].$
\end{theorem}

\begin{definition}\label{def:8}
(G-normal distribution) Let $\mathcal{P}_{Z}$ be a set of probability measures defined on the space $(\Omega,\mathcal{F})$. A measurable
function $Z: \Omega \mapsto \mathbb{R}$ is said to follow a G-normal distribution with lower variance $\underline{\sigma}^{2}$ and upper variance $\overline{\sigma}^{2}$,  $0<\underline{\sigma}\leq\overline{\sigma}$, if for every Lipschitz function $\varphi$,
$$
\hat{\mathbb{E}}[\varphi(Z)]=\sup _{\bbP  \in \mathcal{P}_Z} E_P[\varphi(Z)]=\sup _{\bbP  \in \mathcal{P}_Z} \int_{\Omega} \varphi(Z) d \bbP =u(1,0 ; \varphi),
$$
where $\{u(t,x; \varphi):(t,x)\in [0,\infty)\times\mathbb{R}\}$ is the unique viscosity solution to the Cauchy problem,

$$
u_t=\frac{1}{2}\left(\bar{\sigma}^2\left(u_{x x}\right)^{+}-\underline{\sigma}^2\left(u_{x x}\right)^{-}\right), \quad u(0, x)=\varphi(x).
$$
In the above expression, $u_t=\partial u/\partial t$, $u_{xx}=\partial^{2} u/\partial x^{2}$, and the superscripts $+$ and $-$ denote the positive and negative parts respectively.
\end{definition}

\begin{remark}

\begin{enumerate}[(i)]
  \item For G-normally distributed $Z$, we can prove that $\hat{\mathbb{E}}[Z]=\hat{\mathbb{E}}[-Z]=0$. Therefore $Z$ has no mean-uncertainty.
  \item $\overline{\sigma}^{2}=\hat{\mathbb{E}}[Z^{2}],\ \underline{\sigma}^{2}=-\hat{\mathbb{E}}[-Z^{2}]$.
  \item When $Z$ obeys the G-normal distribution, we say $Z\sim\mathcal{N}(0,[\underline{\sigma}^{2},\overline{\sigma}^{2}])$.
  \item As expected, when $\underline{\sigma}=\overline{\sigma}=\sigma$, the G-normal distribution reduces to the classical normal distribution $\mathcal{N}(0,\sigma^{2})$.
\end{enumerate}
\end{remark}

\begin{theorem}
(Central limit theorem) Let $\{X_{i}\}_{i=1}^{\infty}$ be a sequence of i.i.d. random variables on a sublinear expectation space $(\Omega,\mathcal{H},\hat{\mathbb{E}})$. We assume that $\hat{\mathbb{E}}[X_{1}]=-\hat{\mathbb{E}}[-X_{1}]=\mu$ and $\hat{\mathbb{E}}[|X_{1}|^{3}]<\infty$. Then for any Lipschitz function $\varphi$,
$$
\lim_{n\rightarrow\infty} \hat{\mathbb{E}}[\varphi(\sqrt{n}(\bar{X}_{n}-\mu))]=u(1,0 ; \varphi),
$$
where $u(1,0 ; \varphi)$ is  given in Definition~\ref{def:8} with
$$
\overline{\sigma}^{2}=\hat{\mathbb{E}}[(X_{1}-\mu)^{2}],\ \underline{\sigma}^{2}=-\hat{\mathbb{E}}[-(X_{1}-\mu)^{2}].
$$
\end{theorem}

\section{Two variants of the significance test}\label{sec:B}

In Section~\ref{sec:2}, we proposed a robust procedure for a typical form of the classical univariate significance test.
Here, we address two variants of the test  under sublinear expectation. All technical proofs are omitted due to space limitations, which are similar to those given in Sections~\ref{sec:2} and \ref{sec:3}.

\subsection{First variant of the significance test}
Suppose that the observed data sequence $X_{1},\ldots,X_{n}$ is generated by the data generation process [*]  in Section~\ref{sec:2}, and we are interested in testing about their common mean $\mu$:
\begin{equation}
(\tII)\quad
  \quad H_{0}:\mu\geq\mu_{0}  \quad \quad \text{versus} \quad \quad H_{1}:\mu<\mu_{0},
\end{equation}
where $\mu_{0}$ is a prespecified constant.

When $\mu=\mu_{0}$, according to Theorem \ref{thm:1}, we have
\begin{equation}\label{eq:11}
\begin{split}
p_{2}(c ; \underline{\sigma}, \overline{\sigma}):&=\lim _{n \rightarrow \infty} \sup _{\left\{\sigma_{i}\right\} \in \Sigma(\underline{\sigma}, \overline{\sigma})} \bbE \left[\mathbbm{1}\left(\sqrt{n} (\bar{X}_{n}-\mu_{0})<c\right)\right] \\
 &=\frac{2}{\overline{\sigma}+\underline{\sigma}} \int_{-\infty}^{c}\{\phi(z / \overline{\sigma}) \mathbbm{1}(z \leq 0)+\phi(z / \underline{\sigma}) \mathbbm{1}(z>0)\} d z.
\end{split}
\end{equation}
For any $0<\underline{\sigma}\leq\overline{\sigma}<\infty$, it is easy to see that $p_{2}(c ; \underline{\sigma}, \overline{\sigma})$ is a monotonically increasing function of $c$ with
$$
\lim _{c \rightarrow -\infty}p_{2}(c ; \underline{\sigma}, \overline{\sigma})=0 \qquad\text{and}\qquad \lim _{c \rightarrow \infty}p_{2}(c ; \underline{\sigma}, \overline{\sigma})=1.
$$
For a given significance level $\alpha\ (0<\alpha<0.5)$, let $p_{2}(c ; \underline{\sigma}, \overline{\sigma})=\alpha$, i.e.,
$$
\frac{2}{\overline{\sigma}+\underline{\sigma}} \int_{-\infty}^{c}\phi(z / \overline{\sigma}) d z=\alpha.
$$
It follows that
$$
c=\overline{\sigma}\Phi^{-1}\left(\frac{\alpha(\overline{\sigma}+\underline{\sigma})}
{2\overline{\sigma}}\right).
$$
The robust rejection region is given as follows:
\begin{equation}
\label{WIIG}
W_{\tII}^{G}=\{\sqrt{n} (\bar{X}_{n}-\mu_{0})<c_{2}\},\ \text{with }\ c_{2}=\overline{\sigma}\Phi^{-1}\left(\frac{\alpha(\overline{\sigma}
+\underline{\sigma})}{2\overline{\sigma}}\right).
\end{equation}
While the traditional rejection region is
\begin{equation}
\label{WII}
W_{{\tII}}=\{\sqrt{n} (\bar{X}_{n}-\mu_{0})<S_{n}\Phi^{-1}(\alpha)\}.
\end{equation}

Suppose that there are two samples with different variances but the same mean, the first sample $\{W_{1i}\}_{i=1}^{n_{1}}$ obeys $\mathcal{N}(\mu,\overline{\sigma}^{2})$ and the second sample $\{W_{2i}\}_{i=1}^{n_{2}}$ obeys $\mathcal{N}(\mu,\underline{\sigma}^{2})$. Let $n=\min\{n_{1},n_{2}\}$,  for a given constant $c$ the  experimenter strategically chooses $Z_i=W_{1i}$ or $Z_i=W_{2i}$ at the $i$th stage to maximize the probability of $\{\sqrt{n} (\bar{Z}_{n}-\mu_{0})<c\}$ when $\mu=\mu_{0}$ (i.e., maximize the actual type I error rate). Until the final stage $n$, the experimenter follows a sequential strategy $\theta=(\sigma_{1},\ldots,\sigma_{n})$, where $\sigma_{i}=\overline{\sigma}$ indicates $Z_i=W_{1i}$ and $\sigma_{i}=\underline{\sigma}$ indicates $Z_i=W_{2i}$. That is,
$$
Z_i=\left\{\begin{array}{l}
W_{1i}, \text { if } \sigma_{i}=\overline{\sigma}, \\[1mm]
W_{2i}, \text { if } \sigma_{i}=\underline{\sigma},
\end{array}\right.
$$
where the $\{\sigma_{i}\}_{i=1}^{n}$ are practically identified using historical information.

The next theorem shows the optimal manipulation strategy $\theta^*$ the experimenter will take to maximize the probability of rejecting the null hypothesis when the data is actually generated under the null hypothesis.

\begin{theorem}\label{thm:8}
(Optimal manipulation strategy $\theta^*$ ) For a given constant $c$, we can construct the asymptotically optimal strategy $\theta^*$ to maximize the probability of $\{\sqrt{n} (\bar{Z}_{n}-\mu_{0})<c\}$ when $\mu=\mu_{0}$ (i.e., attain the supremum in \eqref{eq:11}) as follows: for $i\geq2$,
\begin{equation}
\begin{array}{rlrl}
\sigma_{i}=\overline{\sigma}, & & \text { if }\ \frac{\xi_{i-1}}{\sqrt{n}} \geq c, \\
\sigma_{i}=\underline{\sigma}, & & \text { if }\  \frac{\xi_{i-1}}{\sqrt{n}} <c, \\
\end{array}
\end{equation}
where $\xi_i =\sum_{j=1}^i(Z_{j}-\mu_{0})$ and $\sigma_{1}$ can be chosen randomly.
\end{theorem}

The following proposition shows that under such manipulated data,  the traditional procedure with rejection region $W_{{\tII}}$~\eqref{WII} will inflate the rejection probability.

\begin{proposition}
For a given nominal significance level $\alpha\ (0<\alpha<0.5)$, if the experimenter chooses data $Z_{1},\ldots,Z_{n}$ as in Theorem \ref{thm:8} with $c=\bar{\sigma} \Phi^{-1}\left(\alpha\right)$, then the experimenter can increase the actual type I error rate of the traditional rejection region $W_{{\tII}}$~\eqref{WII} to a value larger than $2\alpha\overline{\sigma}/(\underline{\sigma}+\overline{\sigma})$.
\end{proposition}

The subsequent proposition demonstrates that the robust rejection region $W_{{\tII}}^{G}$ can effectively control the type I error rate.

\begin{proposition}
For a given nominal significance level $\alpha\ (0<\alpha<0.5)$, even if the experimenter chooses data $Z_{1},\ldots,Z_{n}$ according to Theorem \ref{thm:8} with $c=$ $\overline{\sigma}\Phi^{-1}\left(\frac{\alpha(\overline{\sigma}+\underline{\sigma})}{2\overline{\sigma}}\right)$, the robust rejection region $W_{{\tII}}^{G}$~\eqref{WIIG} given by the sublinear expectation can ensure that the actual type I error rate equals $\alpha$.
\end{proposition}

\subsection{Second variant of the significance test}
Suppose that the observed data sequence $X_{1},\ldots,X_{n}$ is generated by the data generation process [*] defined  in
Section~\ref{sec:2}, and we are interested in the following significance test  about the common mean $\mu$:

\begin{equation}
({\tIII}) \qquad H_{0}:\mu=\mu_{0}  \qquad \text{versus} \qquad
H_{1}:\mu\neq\mu_{0},
\end{equation}
where $\mu_{0}$ is a prespecified constant.

When $\mu=\mu_{0}$, by Corollary 4 of \cite{peng2020hypothesis}, we have
\begin{equation}\label{eq:14}
\begin{split}
p_{3}(c ; \underline{\sigma}, \overline{\sigma}):&=\lim _{n \rightarrow \infty} \sup _{\left\{\sigma_{i}\right\} \in \Sigma(\underline{\sigma}, \overline{\sigma})} \bbE \left[\mathbbm{1}\left(\sqrt{n} |\bar{X}_{n}-\mu_{0}|>c\right)\right] \\
 &\approx2p_{1}(c ; \underline{\sigma}, \overline{\sigma}).
\end{split}
\end{equation}
For a given significance level $\alpha\ (0<\alpha<0.5)$, let $p_{3}(c ; \underline{\sigma}, \overline{\sigma})=\alpha$, we have the following robust rejection region
\begin{equation}
    \label{WIIIG}
W_{{\tIII}}^{G}=\{\sqrt{n} |\bar{X}_{n}-\mu_{0}|>c_{3}\},\ \text{with }\ c_{3}=\overline{\sigma}\Phi^{-1}\left(1-\frac{\alpha(\overline{\sigma}
+\underline{\sigma})}{4\overline{\sigma}}\right).
\end{equation}
While the traditional rejection region is
\begin{equation}
    \label{WIII}
W_{{\tIII}}=\{\sqrt{n} |\bar{X}_{n}-\mu_{0}|>S_{n}\Phi^{-1}(1-\alpha/2)\}.
\end{equation}

Assume that there are two samples with different variances but the same mean, the first sample $\{W_{1i}\}_{i=1}^{n_{1}}$ obeys $\mathcal{N}(\mu,\overline{\sigma}^{2})$ and the second sample $\{W_{2i}\}_{i=1}^{n_{2}}$ obeys $\mathcal{N}(\mu,\underline{\sigma}^{2})$. Let $n=\min\{n_{1},n_{2}\}$,  for a given constant $c$ the  experimenter strategically chooses $Z_i=W_{1i}$ or $Z_i=W_{2i}$ at the $i$th stage to maximize the probability of $\{\sqrt{n} |\bar{Z}_{n}-\mu_{0}|>c\}$ when $\mu=\mu_{0}$ (i.e., maximize the actual type I error rate). Until the final stage $n$, the experimenter follows a sequential strategy $\theta=(\sigma_{1},\ldots,\sigma_{n})$, where $\sigma_{i}=\overline{\sigma}$ indicates $Z_i=W_{1i}$ and $\sigma_{i}=\underline{\sigma}$ indicates $Z_i=W_{2i}$. That is,
$$
Z_i=\left\{\begin{array}{l}
W_{1i}, \text { if } \sigma_{i}=\overline{\sigma}, \\[1mm]
W_{2i}, \text { if } \sigma_{i}=\underline{\sigma},
\end{array}\right.
$$
where the $\{\sigma_{i}\}_{i=1}^{n}$ are practically identified using historical information.

The subsequent theorem reveals the optimal manipulation strategy $\theta^*$ the experimenter will take to maximize the probability of rejecting the null hypothesis when the data is actually generated under the null hypothesis.
\begin{theorem}\label{thm:9}
(Optimal manipulation strategy $\theta^*$ ) For a given constant $c$, we can construct the asymptotically optimal strategy $\theta^*$ to maximize the probability of $\{\sqrt{n} |\bar{Z}_{n}-\mu_{0}|>c\}$ when $\mu=\mu_{0}$ (i.e., attain the supremum in \eqref{eq:14}) as follows: for $i\geq2$,
\begin{equation}
\begin{array}{rlrl}
\sigma_{i}=\overline{\sigma}, & & \text { if }\ \frac{|\xi_{i-1}|}{\sqrt{n}}\leq c, \\
\sigma_{i}=\underline{\sigma}, & & \text { if }\  \frac{|\xi_{i-1}|}{\sqrt{n}} >c, \\
\end{array}
\end{equation}
where $\xi_i =\sum_{j=1}^i(Z_{j}-\mu_{0})$ and $\sigma_{1}$ can be chosen randomly.
\end{theorem}

In the next proposition, we uncover that the traditional rejection region $W_{{\tIII}}$~\eqref{WIII} is ineffective in controlling the type I error rate.

\begin{proposition}
For a given nominal significance level $\alpha\ (0<\alpha<0.5)$, if the experimenter chooses data $Z_{1},\ldots,Z_{n}$ as in Theorem \ref{thm:9} with $c=\bar{\sigma} \Phi^{-1}\left(1-\alpha/2\right)$, then the experimenter can increase the actual type I error rate of the traditional rejection region $W_{{\tIII}}$~\eqref{WIII} to a value larger than  $2\alpha\overline{\sigma}/(\underline{\sigma}+\overline{\sigma})$.
\end{proposition}

The excellence of the robust rejection region $W_{{\tIII}}^{G}$~\eqref{WIIIG}  in controlling the type I error rate is confirmed by the following proposition.

\begin{proposition}
For a given nominal significance level $\alpha\ (0<\alpha<0.5)$, even if the experimenter chooses data $Z_{1},\ldots,Z_{n}$ according to Theorem \ref{thm:9} with $c=$ $\overline{\sigma}\Phi^{-1}\left(1-\frac{\alpha(\overline{\sigma}+\underline{\sigma})}{4\overline{\sigma}}\right)$, the robust rejection region $W_{{\tIII}}^{G}$~\eqref{WIIIG} given by the sublinear expectation can ensure that the actual type I error rate does not exceed $\alpha$.
\end{proposition}

\newpage
\bibliographystyle{plainnat}
\bibliography{reference}

\begin{thebibliography}{16}
\providecommand{\natexlab}[1]{#1}
\providecommand{\url}[1]{\texttt{#1}}
\expandafter\ifx\csname urlstyle\endcsname\relax
  \providecommand{\doi}[1]{doi: #1}\else
  \providecommand{\doi}{doi: \begingroup \urlstyle{rm}\Url}\fi

\bibitem[Epstein and Ji(2013)]{epstein2013ambiguous}
Larry~G Epstein and Shaolin Ji.
\newblock Ambiguous volatility and asset pricing in continuous time.
\newblock \emph{The Review of Financial Studies}, 26\penalty0 (7):\penalty0
  1740--1786, 2013.

\bibitem[Fang et~al.(2019)Fang, Peng, Shao, and Song]{fang2019limit}
Xiao Fang, Shige Peng, Qi-Man Shao, and Yongsheng Song.
\newblock {Limit theorems with rate of convergence under sublinear
  expectations}.
\newblock \emph{Bernoulli}, 25:\penalty0 2564 -- 2596, 2019.

\bibitem[Head et~al.(2015)Head, Holman, Lanfear, Kahn, and
  Jennions]{head2015extent}
Megan~L Head, Luke Holman, Rob Lanfear, Andrew~T Kahn, and Michael~D Jennions.
\newblock The extent and consequences of p-hacking in science.
\newblock \emph{PLoS biology}, 13\penalty0 (3):\penalty0 e1002106, 2015.

\bibitem[Ji et~al.(2023)Ji, Peng, and Yang]{ji2023imbalanced}
Xuan Ji, Shige Peng, and Shuzhen Yang.
\newblock Imbalanced binary classification under distribution uncertainty.
\newblock \emph{Information Sciences}, 621:\penalty0 156--171, 2023.

\bibitem[Jin and Peng(2021)]{jin2021optimal}
Hanqing Jin and Shige Peng.
\newblock Optimal unbiased estimation for maximal distribution.
\newblock \emph{Probability, Uncertainty and Quantitative Risk}, 6\penalty0
  (3):\penalty0 189--198, 2021.

\bibitem[Lin et~al.(2016)Lin, Shi, Wang, and Yang]{lin2016k}
Lu~Lin, Yufeng Shi, Xin Wang, and Shuzhen Yang.
\newblock k-sample upper expectation linear regression¡ªmodeling,
  identifiability, estimation and prediction.
\newblock \emph{Journal of Statistical Planning and Inference}, 170:\penalty0
  15--26, 2016.

\bibitem[Moss and De~Bin(2023)]{moss2023modelling}
Jonas Moss and Riccardo De~Bin.
\newblock Modelling publication bias and p-hacking.
\newblock \emph{Biometrics}, 79\penalty0 (1):\penalty0 319--331, 2023.

\bibitem[Peng(2004)]{peng2004filtration}
Shige Peng.
\newblock Filtration consistent nonlinear expectations and evaluations of
  contingent claims.
\newblock \emph{Acta Mathematicae Applicatae Sinica, English Series},
  20:\penalty0 191--214, 2004.

\bibitem[Peng(2005)]{peng2005nonlinear}
Shige Peng.
\newblock Nonlinear expectations and nonlinear markov chains.
\newblock \emph{Chinese Annals of Mathematics}, 26\penalty0 (02):\penalty0
  159--184, 2005.

\bibitem[Peng(2008)]{peng2008new}
Shige Peng.
\newblock A new central limit theorem under sublinear expectations.
\newblock \emph{arXiv preprint arXiv:0803.2656}, 2008.

\bibitem[Peng(2010)]{PengICM2010}
Shige Peng.
\newblock Backward stochastic differential equation, nonlinear expectations and
  their applications.
\newblock In \emph{ICM 2010}, Hyderabad, India, 2010.

\bibitem[Peng(2019)]{peng2019nonlinear}
Shige Peng.
\newblock \emph{Nonlinear expectations and stochastic calculus under
  uncertainty: with robust CLT and G-Brownian motion}, volume~95.
\newblock Springer Nature, 2019.

\bibitem[Peng and Zhou(2020)]{peng2020hypothesis}
Shige Peng and Quan Zhou.
\newblock A hypothesis-testing perspective on the {G-normal} distribution
  theory.
\newblock \emph{Statistics and Probability Letters}, 156:\penalty0 108623,
  2020.

\bibitem[Peng et~al.(2023)Peng, Yang, and Yao]{peng2023improving}
Shige Peng, Shuzhen Yang, and Jianfeng Yao.
\newblock Improving value-at-risk prediction under model uncertainty.
\newblock \emph{Journal of Financial Econometrics}, 21\penalty0 (1):\penalty0
  228--259, 2023.

\bibitem[Rokhlin(2015)]{rokhlin2015central}
Dmitry~B Rokhlin.
\newblock Central limit theorem under uncertain linear transformations.
\newblock \emph{Statistics and Probability Letters}, 107:\penalty0 191--198,
  2015.

\bibitem[Simmons et~al.(2011)Simmons, Nelson, and Simonsohn]{simmons2011false}
Joseph~P Simmons, Leif~D Nelson, and Uri Simonsohn.
\newblock False-positive psychology: Undisclosed flexibility in data collection
  and analysis allows presenting anything as significant.
\newblock \emph{Psychological science}, 22\penalty0 (11):\penalty0 1359--1366,
  2011.

\end{thebibliography}

\end{document}